\documentclass[11pt]{article}
\usepackage{amsmath,amsthm, amssymb, amsfonts}
\usepackage[dvips]{graphicx}
\usepackage{enumitem}

\newtheorem{theorem}{Theorem}[section]
\newtheorem{cor}[theorem]{Corollary}
\newtheorem{lem}[theorem]{Lemma}
\newtheorem{Proposition}[theorem]{Proposition}
\newtheorem{Definition}[theorem]{Definition}
\newtheorem{rem}[theorem]{Remark}
\newtheorem{ex}[theorem]{Example}

\newcommand{\ra}{\rightarrow}
\newcommand{\ds}{\displaystyle}

\newcommand{\cX}{{\mathcal X}}
\newcommand{\T}{{\mathbb T}}%
\newcommand{\R}{\mathbb R}%
\newcommand{\C}{\mathbb C}%
\newcommand{\Z}{\mathbb Z}%
\newcommand{\tree}{{\mathbf T}}
\newcommand{\path}{{\mathbf P}}
\newcommand{\1}{\mathbf 1}

\title{Combinatorial Heat and Wave Equations on Certain Classes of Infinite Cayley and Coset Graphs}
\author{S. Mohanty$^*$ \quad and \quad A. K. Lal\footnote{Department of Mathematics and Statistics,
IIT Kanpur,
Kanpur, India - 208016. \newline \indent Emails:
sumitmth@iitk.ac.in, \quad arlal@iitk.ac.in  \quad FAX No:
91-512-2597500. }}
\date{}
\begin{document}
\maketitle


\begin{abstract}
The combinatorial heat and wave equations  on all finite Cayley and coset graphs
with discrete time variable was solved by Lal {\it et al}. In this paper,
the results of the above paper are extended for infinite Cayley and coset graphs,  whenever
the associated groups  are discrete, abelian and finitely generated. Furthermore,
we study the solution of the combinatorial heat and wave equations on a $k$-regular
tree whose associated group is a non-abelian free group on $k$  generators, each
of order~$2$. It turns out that in case of Cayley graphs the solutions to combinatorial heat and wave equations
are  weighted sum of the initial functions  over balls of certain radius which are dependent
on the discrete time variable.

\end{abstract}
\noindent {\sc\textsl{Keywords}:} combinatorial Laplacian, combinatorial heat equation, combinatorial wave equation, Cayley graph,  $k$-regular tree.

\section{Introduction}

Let $G=(V(G),E(G))$, in short $(V,E)$, be  an undirected connected  graph without loops or multiple edges, with $V$ as the set of vertices and $E$
as the set of edges in $G.$  We write $x\sim y$ to indicate that an
undirected edge $\{x,y\}\in E$, {\it i.e.}, the vertices $x,y \in V$ are adjacent in $G.$  The degree of the vertex $x,$ denoted by $m(x),$ is the number of vertices in $V$  that are adjacent to $x.$ A graph is  said to be $k$-regular if $m(x)=k$ for all $x \in V$ and is called  locally finite if $m(x)<\infty$ for all $x\in V$.

A connected graph $G$  without  loops or multiple edges is a metric space with respect to the metric $d_G(x,y)$, the length
of the shortest path from $x$ to $y$ for all  $x, y \in G$.

Let $G=(V,E)$ be a connected locally finite graph. For $1\leq p<\infty,$ let us consider the normed linear space
$ L^p(V)= \{\ f:V\ra \C  : \ \|f\|_{L^p(V)} < \infty \}$, where
$\|f\|_{L^p(V)}= \left(\sum_{x\in V} |f(x)|^p \right)^{\frac{1}{p}}.$
Note that, for  $p=2$, $L^2(V)$ is  a Hilbert space with
$\left<f, g\right>_{L^2(V)} \ = \sum_{x\in   V}f(x)\overline{g(x)}$ as its inner product. Given a function $f: V\ra \C$, the combinatorial Laplacian operator on $G$ is defined by
\begin{equation}\label{eqn:laplacian}
\Delta_G f(x) = m(x)f(x)-\sum_{y\sim x}f(y) \ \mbox{for all}  \  x\in V.
\end{equation}
Observe  that $\Delta_G $ is bounded on $L^2(V)$ if and only if $m(x)$ is uniformly bounded. In case $G$ is a finite graph, the operator $\Delta_G$
represents a $|V|\times |V|$ matrix (where $|S|$ denotes the cardinality of the set $S$), called the  Laplacian matrix of $G$. It is well known that $\Delta_G$ is a positive semi-definite matrix
with $0$ as the smallest eigenvalue. Moreover, if
$0=\lambda_1(G)\leq \lambda_2(G)\leq\cdots \leq \lambda_{|V|}(G)$
are the eigenvalues of $\Delta_G$ then  $\lambda_2(G)>0$
if and only if $G$ is connected (for  details   see~\cite{Bapat}).

\begin{rem}\label{rem:choice-defn-lap}
Recall that, if we define the classical Laplacian operator on $\R^n$  by $\Delta f= -\sum_{i=1}^{n}\frac{\partial^2f}{\partial x_i^2}$ then its eigenvalues $0=\lambda_1\leq \lambda_2\leq \cdots \rightarrow \infty$ form a discrete subset (with multiplicities) of $\R_+=\{x\in \R: x\geq 0\}$
(for details see~\cite[pages~$192$-$195$]{Chavel1}). Therefore
from the previous paragraph it is clear that $\Delta_G$ may be
viewed as the discrete analogue of  the  classical Laplacian $\Delta$.
\end{rem}

We now recall the definition of  a Cayley and coset graph.

\begin{Definition}[\textbf{Cayley Graph}]
Let $\Gamma$ be a  group with $e_{\Gamma}$ as its identity element. Let $S \subset \Gamma$  such that  $S$ generates $\Gamma$
($\langle S\rangle=\Gamma$), $S=S^{-1},  e_{\Gamma} \notin S$ and $|S| < \infty$. The Cayley graph $Cay(\Gamma,S)$ on $\Gamma$ with respect to $S$ has $V=\Gamma$  and  $E= \left \{ \{x, xs\}\ :  x \in \Gamma, s \in S \right\}$.
\end{Definition}

\begin{rem}\label{rem:cayley1}
Cayley graphs as defined above are $|S|$-regular,  undirected, connected and have no loops
nor multiple edges (see~\cite{lauri:scap}). The metric $d_G$ on the Cayley graph, $G = Cay(\Gamma,S)$,   is given by
$$d_G(x,y)= \min  \{ r  : \ y=x s_{i_1} s_{i_2} \cdots
s_{i_r}, \ \text{where} \ s_{i_1}, s_{i_2},\ldots, s_{i_r} \in S\}. $$
\end{rem}
\begin{Definition}[\textbf{Coset Graph}]
Let $\Gamma$ be a group, $H$  a subgroup of $\Gamma$ and $S$ be a subset of $\Gamma$
such that $S\subseteq \Gamma \setminus H,$ $S^{-1}= S$ and $H\cup S$ generates $\Gamma.$
The coset graph $G= {\mbox{Coset }}(\Gamma,H, S)$ on $\Gamma$ with respect to $H$ and $S$
is defined to be the graph having $V(G)$ as the set of all distinct left cosets of $H$
in $\Gamma$ and $E(G)= \left \{ \{xH,xHs\}\  : \ x\in \Gamma \ , \ s \in S \right\},$
{\it i.e.}, for  $xH, yH \in V(G),  xH\neq yH,$ $xH\sim yH$ if  $x^{-1}y \in H S H.$
\end{Definition}
Observations similar to that of Cayley graphs in Remark~\ref{rem:cayley1} can also be
made for the coset graphs. The next remark is an important observation on coset graphs.

\begin{rem}\label{rem:coset1}
In the above definition of a coset graph, the set $S$ may have the property that
$H s_i = H s_j$ for two distinct elements $s_i, s_j \in S$. In this case, the
contribution of $s_i$ and $s_j$ to the edge set of the coset graph remains the same.
Therefore, from the set $S$, we extract a set $\widetilde{S}$ such that the elements
of $\widetilde{S}$ give all the distinct right cosets of $H$ in $S$.
\end{rem}

We now  define the difference operator, the discrete time analogue of the differentiation operator.
\begin{Definition}\label{def:diff-opr}
Let $\Z_+=\{0,1,2, \ldots\}.$ Then, for any complex valued function $v:\Z_+\rightarrow \C,$
let $$\partial_n v(n)=v(n+1)-v(n).$$
\end{Definition}

Let  $G=(V,E)$ be a Cayley  graph. Then, in this paper, we are interested in solving the combinatorial heat equation on $G$ given by\begin{equation}\label{eqn:heat-graph}
\begin{split}
\Delta_{G} u(x,n) + & \partial_n u(x,n) =0 {\mbox{ on }} V(G)\times \Z_+,\\
u(x,0)&=f(x),
\end{split}
\end{equation}
and the combinatorial wave equation on $G$ given by
\begin{equation}\label{eqn:wave-graph}
\begin{split}
\Delta_{G} u(x,n)+ &\partial_n^2 u(x,n) =0 {\mbox{ on }} V(G) \times \Z_+,\\
u(x,0)&=f(x),\quad  \partial_n   u(x,0)=g(x). \\
\end{split}
\end{equation}

The above  equations were first studied in~\cite{dragomir} for Hamming graphs
on the vertex set $\Z_2^{N}$. The results of~\cite{dragomir} related to the
above equations were generalized for all finite Cayley and coset graphs ({\it i.e.},
for all finite vertex transitive graphs) in~\cite{Lal}. Note that
both~\cite{dragomir} and~\cite{Lal}  used the theory of  Fourier analysis
on finite groups to solve these equations.

In this paper, we use techniques from  Fourier analysis on locally
compact groups, discussed in Section~\ref{sec:fourier},   to extend
the results of~\cite{Lal} to solve the combinatorial heat and wave
equations on infinite Cayley and coset graphs whenever the associated groups
are discrete, abelian and finitely generated (see~Sections~\ref{infi-cayley} and~\ref{infi-coset}).
Finally, in Section~\ref{sec:reg-tree}, we  also solve the combinatorial
heat and wave equations on $k$-regular trees which is a Cayley graph having
a non-abelian free group on $k$ generators, each of order $2$, as its associated group.

\subsection{Fourier Transform on Locally Compact Abelian Groups } \label{sec:fourier}
A group $\Gamma$ is said to be a locally compact abelian group if $\Gamma$ is an
abelian group and  there exists a topology on $\Gamma$  with respect to which  $\Gamma$
is locally compact. An important result  on locally compact
abelian groups is stated next.

\begin{Proposition}\cite{Folland1,Rudin}\label{prop:haar-mre}
Let $\Gamma$ be a locally compact abelian group. Then
there exists a nonnegative
regular {\it translation-invariant} measure $m_{\Gamma}$, called the
{\it Haar measure} on $\Gamma$, {\it i.e.}, for every $x\in \Gamma$
and every Borel set $Y$ in $\Gamma$, one has $m_{\Gamma}(xY)=m_{\Gamma}(Y)$.
\end{Proposition}

Now, recall that a {\it character} of a locally compact abelian group $\Gamma$ is a continuous
group homomorphism $\gamma : \Gamma \rightarrow \mathbb{T}$, where $\mathbb{T}$ is the unit circle in $\C$.  Then, the characters of $\Gamma$ form an abelian  group $\widehat{\Gamma}$, called the dual group of $\Gamma$,  with binary operation $(\gamma_1\circ\gamma_2)(x)=\gamma_1(x)\gamma_2(x)$ and the trivial character $\gamma_0$ as its identity element.
Furthermore, there exists a topology on $\widehat{\Gamma}$ with respect to which $\widehat{\Gamma}$ is locally compact. Hence,
using Proposition~\ref{prop:haar-mre}, there exists a Haar measure $m_{\widehat{\Gamma}}$
on $\widehat{\Gamma}$. For more details on the dual group $\widehat{\Gamma}$ and
its topology, the readers can refer to~\cite{Folland1,Rudin,Sugiura}.
We now state a result which gives an interesting relation between the
topology of $\Gamma$ and that of $\widehat{\Gamma}$.

\begin{Proposition}\cite{Rudin}\label{prop:dis-comp}
If $\Gamma$ is a discrete abelian group then the dual group $\widehat{\Gamma}$  is a compact abelian group. In case $\Gamma$ is a compact abelian group then $\widehat{\Gamma}$
is a discrete  abelian group.
\end{Proposition}

\begin{rem}\label{rem:dis-comp-mre}
Given a discrete abelian group $\Gamma$, the corresponding Haar measure
$m_{\Gamma}$ is a counting measure. The dual group $\widehat{\Gamma}$ is a
compact group and hence  the corresponding Haar measure $m_{\widehat{\Gamma}}$
on $\widehat{\Gamma}$  is a finite measure, {\it i.e.}, $m_{\widehat{\Gamma}}(\widehat{\Gamma})<\infty$.
Hence, for $1\le p<\infty,$

$$L^p(\Gamma)=\{f: \Gamma\rightarrow \C :\|f\|_{L^p(\Gamma)}<\infty\}, \;\;\;\; L^p(\widehat{\Gamma})
  =\{f: \widehat{\Gamma}\rightarrow \C :\|f\|_{L^p(\widehat{\Gamma})}<\infty\},$$
where
$\|f\|_{L^p(\Gamma)}=\left(\sum\limits_{x\in \Gamma}|f(x)|^p\right)^{1/p}$
and
$ \|f\|_{L^p(\widehat{\Gamma})}= \left(\frac{1}{m_{\widehat{\Gamma}}(\widehat{\Gamma})}
        \int_{\widehat{\Gamma}} |f(\gamma)|^p \  dm_{\widehat{\Gamma}}(\gamma)\right)^{1/p}.$
\end{rem}

We are now ready to  state  a well known result which will be used  subsequently.
\begin{Proposition}\cite[page~$10$]{Rudin}\label{prop:char-orthogonal}
Let $\Gamma$ be a discrete abelian group. Then
$$\frac{1}{m_{\widehat{\Gamma}}(\widehat{\Gamma})}
                      \int_{\widehat{\Gamma}} \gamma(x) \  dm_{\widehat{\Gamma}}(\gamma)
                             =\left\{\begin{array}{ll}
                                            0, & {\mbox{ if }}  x\neq e_{\Gamma}, \\
                                            1, &  {\mbox{ if }}  x = e_{\Gamma}.
                                      \end{array}\right.$$
\end{Proposition}
With the above background, we recall the following definitions.

\begin{Definition}\label{defn:fourier-dis-gr}
Let $\Gamma$ be a discrete abelian group. Then, the Fourier transform of $f\in L^{1}(\Gamma)$ is given by
$$(\mathfrak{F} f)(\gamma)=\widehat{f}(\gamma)=\sum_{x\in \Gamma}f(x)\gamma(x) \;\; {\mbox{ for each character }} \gamma.$$
Further, if  $\widehat{f}\in L^1(\widehat{\Gamma})$ then,  its inverse Fourier transform is given by
$$\mathfrak{F^{-1}}(\widehat{f})(x)= f(x)= \frac{1}{m_{\widehat{\Gamma}}(\widehat{\Gamma})}
    \int_{\widehat{\Gamma}}\widehat{f}(\gamma)\gamma(x^{-1}) \  dm_{\widehat{\Gamma}}(\gamma) {\mbox{ for each }} x \in \Gamma.$$
\end{Definition}

\begin{rem}\label{rem:density}
Recall that $L^{1}(\Gamma)$ is dense in $L^{2}(\Gamma)$ whenever $\Gamma$ is a discrete abelian group. Further, by Plancherel Theorem~\cite{Rudin}, the notion of Fourier transform can be extended to $L^{2}(\Gamma)$ using the density argument.
\end{rem}

Further recall that for  a  compact abelian
group $\Gamma$ and for a finite number of characters
$\gamma_1,\gamma_2,\ldots,\gamma_t \in \widehat{\Gamma}$,  a function of the form $ P(x)=\sum_{i=1}^{t} b_i \gamma_i(x),$ where $ x\in \Gamma$ and $b_i \in \C$, is called a trigonometric polynomial on $\Gamma$. Now we state a result that  is an application of  the Stone-Weierstrass theorem.

\begin{Proposition}\cite{Rudin}\label{prop:stone}
Let $\Gamma$ be a  compact abelian group. Then the trigonometric polynomials on $\Gamma$ form a dense subalgebra of $\mathcal{C}(\Gamma)$, the set of all
continuous functions on~$\Gamma$.
\end{Proposition}

\begin{rem}\label{rem:hat-f-cont}
Note that if $\Gamma$ is a discrete abelian group then by Proposition~\ref{prop:dis-comp},
the dual group $\widehat{\Gamma}$ is compact. Thus,  Pontryagin Duality Theorem (see~\cite[page~$27$]{Rudin}) implies  that the trigonometric polynomials on $\widehat{\Gamma}$ are of the
form $P(\gamma)=\sum_{i=1}^{n} b_i \gamma(x_i)$, where $x_i \in \Gamma$,
for $1 \le i \le n$. If $f\in L^{1}(\Gamma)$ then by Proposition~\ref{prop:stone},
 $\sum_{x\in \Gamma}f(x)\gamma(x)$, the Fourier
transform of $f$,  converges uniformly to $\widehat{f}$. Hence $\widehat{f}$ is  continuous on $\widehat{\Gamma}$ and  $\widehat{f}\in L^{1}(\widehat{\Gamma})$ as  $\widehat{\Gamma}$ is compact.
\end{rem}

\begin{ex}
For $1\le p<\infty$ and  $\Gamma=\Z$, recall that
$L^p(\Z)=\{f: \Z\rightarrow \C :\|f\|_{L^p(\Z)}<\infty\}$ with $ \|f\|_{L^p(\Z)}=\left(\sum_{r\in \Z}|f(r)|^p\right)^{1/p}.$
Also, recall that $\mathbb{T}=\{z\in \C: |z|=1\} \cong {\R}/{2\pi \Z},$ where $\R$ is the additive group of real numbers. Hence the
functions on $\T$ are identifiable with  $2\pi$-periodic functions on $\R$ and thus, for
$1\le p<\infty,$
$$L^p(\T)=\{f: \T\rightarrow \C :\|f\|_{L^p(\T)}<\infty\} {\mbox{ with }}
\|f\|_{L^p(\T)}=\left(\frac{1}{2\pi}\int_{0}^{2\pi}|f(t)|^p dt\right)^{1/p},$$
where $dt$ is the Lebesgue measure on the interval $[0,2\pi).$

Then, it can be easily deduced that $\widehat{\Gamma} = \widehat{\Z} \cong \T$, {\it i.e.}, the set of characters can be
parameterized as $\{\gamma_t : t\in [0,2\pi)\}$ with  $\gamma_t(r)=e^{irt}$ for all $r\in \Z$ (for details, see~\cite{Rudin,Sugiura}).  Thus, the Fourier transform of
$f\in {L^1(\Z)}$ is given by
$$(\mathfrak{F} f)(\gamma_t)= (\mathfrak{F} f)(t)=\widehat{f}(t)=\sum_{r\in \Z}f(r)\gamma_t(r)
=\sum_{r\in \Z}f(r)e^{irt}.$$
By Remark~\ref{rem:hat-f-cont}, one also has $\hat{f}\in L^1(\T)$  and hence its inverse Fourier transform is given by
$$(\mathfrak{F^{-1}}(\widehat{f}))(r)= f(r)
= \frac{1}{2\pi} \int_{0}^{2\pi}\widehat{f}(t)\; \gamma_t(-r) dt
= \frac{1}{2\pi}\int_{0}^{2\pi}\widehat{f}(t)\; e^{-irt} dt.$$
Hence, note that  the Fourier transform on $\Z$ is nothing but the Fourier series.
\end{ex}

\begin{rem}\label{rem:four1}
Let ${\mathcal P}_N(t)=\sum\limits_{r=-N}^N a_r e^{irt}$ be  a trigonometric polynomial on $\T$. Then, by  Proposition~\ref{prop:char-orthogonal} on $\Z$, we have
$ \frac{1}{2\pi}\int_{0}^{2\pi}e^{irt} dt=
\left\{\begin{array}{ll}  1, & {\mbox{ if }}  r= 0, \\
                          0, & {\mbox{ if }} r\neq 0.
                          \end{array}\right.$
Thus, $ \mathfrak{F^{-1}}({\mathcal P}_N)(r)=\left\{\begin{array}{ll}                                        a_r, & {\mbox{ if }}  |r|\leq N, \\
0, & {\mbox{ otherwise}}.
  \end{array}\right.$
\end{rem}

The convolution of two functions, which in a sense replaces the idea of point wise multiplication of two functions, plays a crucial role in Fourier analysis and is recalled next.
\begin{Definition}\label{def:convolution}
Let  $\Gamma$ be a finite group and  $f,g \in L^{1}(\Gamma)$.
The convolution $f*g$ is defined by
$f*g(x)=\sum_{y\in \Gamma}\ f(y)\; g(x y^{-1}).$
\end{Definition}

We now state few properties of Fourier transform that will be
referred  in subsequent results.
\begin{Proposition}\label{prop:fourier-dis-gr}
Let $\Gamma$ be a discrete abelian group and let $f,g \in L^1(\Gamma)$. Then 
\begin{enumerate}[label={(\arabic*)}]
     \item  $\mathfrak{F}(f*g)=(\mathfrak{F}f) \; (\mathfrak{F}g).$
     \item  for each fixed $y \in \Gamma$, if  $f_y(x)=f(xy^{-1})$ for each $x \in \Gamma$ then $(\mathfrak{F} f_y)(\gamma)=\gamma(y) \; (\mathfrak{F} f)(\gamma).$
     \item  $\mathfrak{F^{-1}}(\1)=\chi_{\{e_{\Gamma}\}},$ where $\1$ is the constant function on $\widehat{\Gamma}$ that takes the value $1$.
\end{enumerate}
\end{Proposition}
\begin{proof}
See~\cite[Theorem~$1.2.4$, page~$9$]{Rudin}, for the proofs of first and second parts.  For the  third part, note that by
Remark~\ref{rem:dis-comp-mre}, the Haar measure $m_{\widehat{\Gamma}}$ on the dual
group $\widehat{\Gamma}$ is a finite measure and hence $\1 \in L^1(\widehat{\Gamma})$.
Thus, using Proposition~\ref{prop:char-orthogonal}, we have
\begin{eqnarray*}
\mathfrak{F^{-1}}(\1)(x)
=\frac{1}{m_{\widehat{\Gamma}}(\widehat{\Gamma})}
    \int_{\widehat{\Gamma}}\1(\gamma)\gamma(x^{-1}) \  dm_{\widehat{\Gamma}}(\gamma)
=\frac{1}{m_{\widehat{\Gamma}}(\widehat{\Gamma})}
    \int_{\widehat{\Gamma}}\gamma(x^{-1}) \  dm_{\widehat{\Gamma}}(\gamma)
= \chi_{\{e_{\Gamma}\}}(x)
\end{eqnarray*}
and hence the required result follows.
\end{proof}

\section{Results on Some Infinite Cayley  Graphs whose Associated Group is Abelian and Finitely Generated}\label{infi-cayley}

Let $\Gamma$ be  an infinite discrete  abelian group generated by
finitely many generators. In this section, we solve the combinatorial heat and wave equations on the infinite Cayley graph $G={\mbox{Cay }}(\Gamma, S)$, whenever $|S|< \infty$.
We start with the combinatorial heat equation.

\begin{theorem}\label{thm:heat11}
Let $\Gamma$ be an infinite discrete abelian group and let  $S = \{s_1,  \ldots, s_k\} \subset \Gamma$
such that $e_{\Gamma} \notin S$, $\langle S\rangle=\Gamma$ and $S=S^{-1}$. Also, let $G={\mbox{Cay }}(\Gamma,S)$ be the Cayley graph on $V(G)=\Gamma$ with respect to $S$. Then, for $f\in L^2(\Gamma)$ the combinatorial heat
equation~(\ref{eqn:heat-graph}) on $G$ admits a unique solution
 $u(x,n)\ =\  K_n*f(x),$ where
$$K_n(x)\ = \ \chi_{\{e_{\Gamma}\}}(x) \ + \
       \sum_{j=1}^n \ (-1)^j {n \choose j}
          \underbrace{\mathfrak{F}^{-1}(a)*\cdots*\mathfrak{F}^{-1}(a)}_{j \ times}(x)$$
with $\mathfrak{F}^{-1}(a)(x) = |S| \; \chi_{\{e_{\Gamma}\}}(x) - \chi_B(x)$ and
$B\subset \Gamma$ is the boundary of the unit ball centered at $e_{\Gamma}$.
\end{theorem}

Before coming to the proof of Theorem~\ref{thm:heat11}, we state a result that
gives information about the solution of~(\ref{eqn:heat-graph})
on the Cayley graph (as stated in Theorem~\ref{thm:heat11}) whenever it exists.

\begin{lem}\label{lem:sol-heat-u}
Suppose the hypothesis of Theorem~\ref{thm:heat11} holds and  $f\in L^2(\Gamma)$. If~(\ref{eqn:heat-graph}) admits a solution $u(x,n)$ on $G={\mbox{Cay }}(\Gamma,S)$ then $u(\cdot,n) \in L^2(\Gamma)$ for all $n\in \Z_+$.
\end{lem}

\begin{proof}
We will  use the principle of mathematical induction on the discrete
time variable $n\in \Z_+$ to prove this result. If $n=0$, then
$u(\cdot,0)=f(\cdot) \in  L^2(\Gamma)$ and hence the result holds
trivially for $n=0$. So, let us assume that $u(\cdot,t) \in L^2(\Gamma)$ for all $t\leq n$ and compute $u(x,n+1)$.

Now, let $t=n+1$. Since $G$ is a $k$-regular graph, using (\ref{eqn:laplacian}), we can re-write (\ref{eqn:heat-graph}) as
$$u(x,n+1)= \sum_{i=1}^k \ u(xs_i, n)-(k-1)\; u(x,n).$$
Using the Cauchy-Schwarz inequality and the induction hypothesis, we have
\begin{eqnarray*}
\|u(\cdot,n+1)\|_{L^{2}(\Gamma)}
           &=&  \Bigg( \sum_{x\in \Gamma}\left|\sum_{i=1}^k \ u(xs_i, n)-(k-1) u(x,n)
                                  \right|^2\Bigg)^{\frac{1}{2}} \\
           & \leq & \sum_{i=1}^k \left( \sum_{x\in \Gamma}|u(xs_i,n)|^2\right)^{\frac{1}{2}}
                   +(k-1)\left( \sum_{x\in \Gamma} |u(x,n)|^2 \right)^{\frac{1}{2}}\\
           &=&(2k-1) \; \|u(\cdot,n)\|_{L^{2}(\Gamma)}< \infty.
\end{eqnarray*}
Thus, $u(\cdot,n+1) \in L^2(\Gamma)$. Hence,  by mathematical
induction, the desired result follows.
\end{proof}

Now, let us  complete the proof of Theorem~\ref{thm:heat11}.


\noindent{\it Proof of Theorem~\ref{thm:heat11}.}
As $G$ is a $k$-regular graph with generating
set $S = \{s_1, \ldots, s_k\}$, using (\ref{eqn:laplacian}), we can re-write~(\ref{eqn:heat-graph}) on $G$ as
\begin{equation}\nonumber
 k\ u(x,n) -\sum_{i=1}^k \ u(xs_i, n) +u(x,n+1)-u(x,n) =0 {\mbox{ and }}  u(x,0) =f(x).
\end{equation}
Now applying the Fourier transform  (see Remark~\ref{rem:density}) and using  Proposition~\ref{prop:fourier-dis-gr}, we have
\begin{equation*}
\widehat{u}(\cdot,n) =  [1-a(\cdot)]^n \widehat{f}(\cdot), \
                          \mbox{where}\  a(\gamma)=  k- \sum\limits_{i=1}^k \ \gamma(s_i^{-1}).
\end{equation*}
By  Lemma~\ref{lem:sol-heat-u} and Plancherel Theorem,  $u(\cdot,n)\in L^2(\Gamma)$  and hence $\widehat{u}(\cdot,n)\in L^{2}(\widehat{\Gamma})$ for all $n\in \Z_+$. Thus, applying
 the inverse Fourier transform and  Proposition~\ref{prop:fourier-dis-gr}, we get  $$ u(x,n)=K_n*f(x),$$ where
$\ds
K_n(x) =  \mathfrak{F}^{-1} \left( (1-a )^n \right)(x)
  = \chi_{\{e_{\Gamma}\}}(x)+  \sum_{j=1}^n\ (-1)^j { n \choose j}
 \underbrace{\mathfrak{F}^{-1}(a)*\cdots*\mathfrak{F}^{-1}(a)}_{j \
        times}(x)$ and
\begin{eqnarray}
\mathfrak{F}^{-1}(a)(x) &=& \frac{1}{m_{\widehat{\Gamma}}
(\widehat{\Gamma})}                               \int_{\widehat{\Gamma}}a(\gamma)
\ \gamma(x^{-1})                                   \  dm_{\widehat{\Gamma}}(\gamma)
                     = \frac{1}{m_{\widehat{\Gamma}}(\widehat{\Gamma})}\int_{\widehat{\Gamma}}
                            \biggl( k -\sum_{j=1}^k \ \gamma(s_j^{-1}) \biggr)
                                \gamma(x^{-1}) \  dm_{\widehat{\Gamma}}(\gamma)\nonumber \\
                     &=& k \ \mathfrak{F}^{-1}(\1)(x)
                              - \sum_{j=1}^k \frac{1}{m_{\widehat{\Gamma}}(\widehat{\Gamma})}
                               \int_{\widehat{\Gamma}}\gamma(s_j^{-1})\ \gamma(x^{-1})
                                  \  dm_{\widehat{\Gamma}}(\gamma)\nonumber \\
                     &=&k \ \chi_{\{e_{\Gamma}\}}(x)
                              - \sum_{j=1}^k \frac{1}{m_{\widehat{\Gamma}}(\widehat{\Gamma})}
                               \int_{\widehat{\Gamma}}\gamma((xs_j)^{-1})
                                  \  dm_{\widehat{\Gamma}}(\gamma). \label{eqn:inva11}
\end{eqnarray}
Note that for $ x = e_{\Gamma}, \;
\mathfrak{F}^{-1}(a)(e_{\Gamma})= k$. But for $x\neq e_{\Gamma}$,
using Proposition~\ref{prop:char-orthogonal}
and (\ref{eqn:inva11}), we have
$\mathfrak{F}^{-1}(a)(x)\neq 0$ if and only if $x\in S$. Thus,
\begin{equation*}
\mathfrak{F}^{-1}(a)(x)\ = k \; \chi_{\{e_{\Gamma}\}}(x) - \chi_S(x)
= k \; \chi_{\{e_{\Gamma}\}}(x) - \chi_B(x)= |S| \; \chi_{\{e_{\Gamma}\}}(x) - \chi_B(x).
\end{equation*}
Since the Fourier inversion is unique, the required result follows. \qed

We now state the main result on the combinatorial wave equation.

\begin{theorem}\label{thm:wave11}
Suppose the hypothesis of Theorem~\ref{thm:heat11} holds. Then, for   $f,g\in L^1(\Gamma)$, the combinatorial wave equation (\ref{eqn:wave-graph}) on the Cayley graph $G={\mbox{Cay }}(\Gamma,S)$
has a solution if and only if $\widehat{g}(\gamma_0)=0$.
Moreover, if this is the case, then this solution is unique and is expressed by
\begin{eqnarray*}
u(x,n)&=& f(x)+ n g(x) + \sum_{i=1}^{[\frac{n}{2}] } \ (-1)^i
{n \choose {2i}} \underbrace{\mathfrak{F}^{-1}(a)*\cdots*\mathfrak{F}^{-1}(a)}_{i \ times}*f(x)\\
  && \hspace{2cm}+\sum_{i=1}^{[\frac{n-1}{2}] } \ (-1)^i {{n}
\choose {2i+1}} \underbrace{\mathfrak{F}^{-1}(a)*\cdots*\mathfrak{F}^{-1}(a)}_{i \ times}*g(x)
\end{eqnarray*}
with $\mathfrak{F}^{-1}(a)(x) = |S| \; \chi_{\{e_{\Gamma}\}}(x) - \chi_B(x)$ and $B\subset \Gamma$ is
the boundary of the unit ball centered at $e_{\Gamma}$.
\end{theorem}

Before proving the above  theorem, we first prove the following result which gives information about the solution of (\ref{eqn:wave-graph}) on $G$, whenever it exists.
The proof is similar to the proof of Lemma~\ref{lem:sol-heat-u} but is presented
here for the sake of completeness.

\begin{lem}\label{lem:sol-wave-u}
Suppose the hypothesis of Theorem~\ref{thm:wave11} holds and $f,g\in L^1(\Gamma)$.
If the combinatorial wave equation (\ref{eqn:wave-graph}) on ${\mbox{Cay }}(\Gamma,S)$ admits a solution $u(x,n)$ then $u(\cdot,n) \in L^1(\Gamma)$ for all $n\in \Z_+$.
\end{lem}

\begin{proof}
As $G$ is a $k$-regular  graph, by using (\ref{eqn:laplacian}), we can re-write  (\ref{eqn:wave-graph}) on $G$ as
\begin{equation}\label{eqn:rec-wave}
\begin{split}
 u(x, n+2) &= 2u(x, n+1)+ \sum_{i=1}^k \ u(xs_i, n)  - (k+1)u(x,n), \\
 u(x,0) &=f(x)  \; {\mbox{ and }} \;    u(x,1)=u(x,0)+g(x).
\end{split}
\end{equation}
We again use the principle of mathematical induction on the discrete time variable $n\in \Z_+$
to complete the proof. Note that if  $n$ is either $0$ or $1$ then from (\ref{eqn:rec-wave}),
we have $u(\cdot,0)=f(\cdot) \in L^1(\Gamma)$ and $u(\cdot,1)=f(\cdot)+g(\cdot) \in  L^1(\Gamma)$.
The result thus holds trivially for $n=0,1$. Assume that the result is true
whenever the discrete time variable $t \leq n+1$, {\it i.e.}, $u(\cdot,t) \in L^1(\Gamma)$
for all $t\leq n+1$.

Now, for $t=n+2$, using (\ref{eqn:rec-wave}), the
triangle inequality and the induction hypothesis, we get
\begin{eqnarray*}
\|u(\cdot,n+2)\|_{L^{1}(\Gamma)}
           & \leq & 2\sum_{x\in \Gamma} |u(x,n+1)|
                    +\sum_{i=1}^k  \sum_{x\in \Gamma}|u(xs_i,n)|+(k+1)\sum_{x\in \Gamma} |u(x,n)|\\
           &=& 2\|u(\cdot,n+1)\|_{L^{1}(\Gamma)}+(2 k+1) \|u(\cdot,n)\|_{L^{1}(\Gamma)}< \infty.
\end{eqnarray*}
Hence, by the principle of mathematical induction, the desired result follows.
\end{proof}

We are now ready to prove  Theorem~\ref{thm:wave11}.


\noindent{\it Proof of Theorem~\ref{thm:wave11}.}
As $G$ is a $k$-regular graph with generating set $S=\{s_1, \ldots, s_k\}$, the application of the  Fourier transform on the combinatorial wave equation (\ref{eqn:wave-graph})  gives
\begin{equation*}
 \widehat{u}(\gamma, n+2)- 2\widehat{u}(\gamma, n+1)-
         \left[-a(\gamma)-1\right]\widehat{u}(\gamma,n)=0, \
 \widehat{u}(\gamma,0)=\widehat{f}(\gamma)  \; {\mbox{ and }} \;
   \widehat{u}(\gamma,1)-\widehat{u}(\gamma,0)=\widehat{g}(\gamma),
\end{equation*}
where $a(\gamma)=  k- \sum\limits_{i=1}^k \ \gamma(s_i^{-1})$. Solving the above recurrence equation, one has
$$\widehat{u}(\gamma, n)=\lambda_\gamma \left(1+\sqrt{-a(\gamma)}\right)^n+ \mu_\gamma \left(1-\sqrt{-a(\gamma)}\right)^n.$$
Now, using the initial conditions, we have
\begin{equation}\label{eqn:wave-initial1}
\lambda_\gamma+\mu_\gamma =\widehat{f}(\gamma) \; {\mbox{ and }}\;
\lambda_\gamma(1+\sqrt{-a(\gamma)})+ \mu_\gamma(1-\sqrt{-a(\gamma)}) =\widehat{f}(\gamma)+\widehat{g}(\gamma).
\end{equation}
Note that,  $a(\gamma)$ is a trigonometric
polynomial on $\widehat{\Gamma}$  and hence continuous. Therefore, using Remark~\ref{rem:hat-f-cont},
$\widehat{f}$ and $\widehat{g}$ are also continuous  functions on $\widehat{\Gamma}$. Thus,
the point-wise estimate of  (\ref{eqn:wave-initial1}) is well defined. Note that  (\ref{eqn:wave-initial1}) is consistent if and only if  $\widehat{g}(\gamma_0)=0$. We also observe that
\begin{equation*}
\widehat{u}(\gamma, n) =  \biggl[ \sum_{i=0}^{[\frac{n}{2}]}{n \choose 2i} (-a)^i(\gamma)\biggr]\widehat{f}(\gamma)
       +  \biggl[ \sum_{i=0}^{[\frac{n-1}{2}]}{ n \choose 2i+1}
          (-a)^i(\gamma)\biggr]\widehat{g}(\gamma).
\end{equation*}
By Lemma~\ref{lem:sol-wave-u}, the solution $u(\cdot,n)\in L^1(\Gamma)$ for all $n\in \Z_+$.
Therefore, using Remark~\ref{rem:hat-f-cont},   $\widehat{u}(\cdot,n)$ is continuous on the
compact group $\widehat{\Gamma}$ and hence $\widehat{u}(\cdot,n)\in L^{1}(\widehat{\Gamma})$.
Taking the inverse Fourier transform and using Proposition~\ref{prop:fourier-dis-gr}, we get
$u(x,n)=F_n*f(x)+ G_n*g(x),$ where
\begin{equation*}
F_n(x) = \chi_{\{e_{\Gamma}\}}(x)  + \sum_{i=1}^{[\frac{n}{2}] } \ (-1)^i \  {n \choose 2i}  \ \underbrace{\mathfrak{F}^{-1}(a)*\cdots*\mathfrak{F}^{-1}(a)}_{i \
 times}(x), {\mbox{ and }}
 \end{equation*}
\begin{equation*}
G_n(x) = n\chi_{\{e_{\Gamma}\}}(x) + \sum_{i=1}^{[\frac{n-1}{2}] } \ (-1)^i \ {n \choose 2i+1} \ \underbrace{\mathfrak{F}^{-1}(a)*\cdots*\mathfrak{F}^{-1}(a)}_{i \ times}(x)
\end{equation*}
with  $\mathfrak{F}^{-1}(a)(x) = |S| \;
\chi_{\{e_\Gamma\}}(x) - \chi_B(x)$ and hence the required result follows. \qed

\begin{rem}
One can verify that given a Cayley graph:
\begin{enumerate}[label={(\arabic*)}]

\item the solution to combinatorial heat equation $u(x,n)$ is a weighted sum of the
      initial function $f$ over the ball of radius of $n$ centered at $x$.
\item the solution to combinatorial wave equation $u(x,n)$ is a weighted sum of the
      initial functions  $f$ and $g$ over the ball of radius of $[\frac{n}{2}]$
      and  $[\frac{n-1}{2}]$  centered at $x$ respectively.
\end{enumerate}
\end{rem}

\section{Results on Some  Infinite Coset  Graphs whose Associated Group  is Abelian and
 Finitely Generated }\label{infi-coset}

Let $\Gamma$ be  an infinite discrete  abelian group generated by finitely many generators and also let $\Gamma$ contain a finite subgroup $H$. In this subsection, we will solve
the combinatorial heat and wave equations on the coset graph
$G={\mbox{Coset }}(\Gamma,H, S)$, whenever the group $\Gamma$
 has the above mentioned property and $|S|< \infty$. Therefore, in this section, we assume that $\Gamma$ is an infinite discrete  abelian group generated by finitely many generators and it also contains a finite subgroup $H$. To proceed further, we also assume that $S$ is  a finite subset of $\Gamma$ such that $S\subset \Gamma \setminus H$, $S^{-1}= S$ and $H\cup S$ generates $\Gamma$.

Proceeding in a manner similar to the case of finite coset graphs,
we  construct a new graph $\widetilde{G}$ with $V(\widetilde{G})=\Gamma$  as the vertex set.
Two elements $x,y \in \Gamma$ are  adjacent  if there exist $b_i H,\; b_j H \in V(G)$ such that
 $x \in b_i H, y \in b_j H$, $ b_i H \neq b_j H\ $ and $b_i H \sim b_j H$. Thus, $\widetilde{G}$
is a $\delta$-regular graph with $\delta= k \ |H|$, where $k = | \widetilde{S} |$ with $\widetilde{S} $
defined as in Remark~\ref{rem:coset1}. Further,
 for any complex valued function $f: V(G)\rightarrow \C$, we
define a function $\widetilde{f}: V(\widetilde{G})\rightarrow \C$ by
\begin{equation}\label{eqn:newf:1}
\widetilde{f}(x) =  f(b H),  {\mbox{ whenever }}  x \in b H,
\end{equation}
{\it i.e.}, fix a left coset $b H$ of $H$ in $\Gamma$ and let $y, z \in
b H$. Then $\widetilde{f}(y) = \widetilde{f}(z)=f(b H)$. Using Theorem~\ref{thm:heat11} and
proceeding as in the proof of \cite[Lemma~$2.3$]{Lal}, we obtain the next lemma.

\begin{lem}\label{lem:convertheat1}
Let $\Gamma$ be an infinite discrete abelian group,  $H$ be a finite subgroup of $\Gamma$ and let
$S$ be  a finite subset of $\Gamma$ such that $S\subset \Gamma \setminus
H$, $S^{-1}= S$ and $H\cup S$ generates $\Gamma$. Let
$G={\mbox{Coset }}(\Gamma,H, S)$ be the coset graph of $\Gamma$ with respect to
$H$ and $S$. Then, for   $\widetilde{f} \in L^2(\Gamma)$, the initial value problem
\begin{equation}\label{eqn:changeheat11}
\begin{split}
\frac{1}{|H|}\Delta_{\widetilde{G}} {\widetilde{u}}(x,n)&+
\partial_n{ \widetilde{u}}(x,n)=0 \quad {\mbox{on }}  V(\widetilde{G})\times \Z_+, \;  \\
\widetilde{u}(x,0)&=\widetilde{f}(x)
\end{split}
\end{equation}
admits a unique solution with $\widetilde{u}(x,n)=
{\mbox{constant}}$ on each left coset of $H$.
\end{lem}
\begin{proof}
Let $\{H s_1, H s_2, \ldots, H s_k \}$ be the set of all  distinct
right cosets of $H$ in $\Gamma$, where $s_i \in S,  1\leq i\leq k$. Using
(\ref{eqn:laplacian}),  we can rewrite (\ref{eqn:changeheat11}) as
\begin{equation*}
\begin{split}
\frac{1}{|H|}\left( k|H|\widetilde{u}(x,n)-\sum_{h\in H} \sum_{i=1}^k \widetilde{u}(x h s_i,n )\right)
                  &+ \widetilde{u}(x,n+1)  - \widetilde{u}(x,n)=0,\\
\widetilde{u}(x,0)&=\widetilde{f}(x).
\end{split}
\end{equation*}
Using an  argument similar to that  in Theorem~\ref{thm:heat11},  we get
\begin{equation}\label{eqn:hat(u11)}
\widehat{\widetilde{u}}(\gamma,n) =  [1-\widetilde{a}(\gamma)]^n \widehat{\widetilde{f}}(\gamma),
\end{equation}
where $ \widetilde{a}(\gamma)=k-\frac{1}{|H|}\sum_{h \in H}
\sum_{i=1}^k\gamma((hs_i)^{-1}).$ Taking the inverse Fourier transform on (\ref{eqn:hat(u11)}), we get $\widetilde{u}(x,n)\ =\widetilde{K}_n*\widetilde{f}(x)$, where
\begin{equation*}
 \widetilde{K}_n(x)=\  \sum_{j=0}^n \ (-1)^j  {n \choose j}
 \underbrace{\mathfrak{F}^{-1}(\widetilde{a})*\cdots*\mathfrak{F}^{-1}
 (\widetilde{a})}_{j \ times}(x).
\end{equation*}
The Fourier inversion of $ \widetilde{a}(\gamma)=k-\frac{1}{|H|}\sum_{h \in H}
\sum_{i=1}^k\gamma((hs_i)^{-1})$ is given by
\begin{eqnarray*}
\mathfrak{F}^{-1}(a)(x)&=& \frac{1}{m_{\widehat{\Gamma}}(\widehat{\Gamma})}\int_{\widehat{\Gamma}}
                            \left( k-\frac{1}{|H|}\sum_{h \in H} \sum_{i=1}^k\gamma((hs_i)^{-1}) \right)
                                \gamma(x^{-1}) \  dm_{\widehat{\Gamma}}(\gamma)\nonumber \\
&=& \hspace{-.1in} k \ \mathfrak{F}^{-1}(\1)(x)
 - \frac{1}{|H|}\sum_{h \in H} \sum_{i=1}^k
 \left(\frac{1}{m_{\widehat{\Gamma}}(\widehat{\Gamma})}
                               \int_{\widehat{\Gamma}}\gamma(x^{-1}) \ \gamma((hs_i)^{-1})
                                  \  dm_{\widehat{\Gamma}}(\gamma)\right)\nonumber \\
 &=& \hspace{-.1in} k \ \mathfrak{F}^{-1}(\1)(x)
 - \frac{1}{|H|}\sum_{h \in H} \sum_{i=1}^k
 \left(\frac{1}{m_{\widehat{\Gamma}}(\widehat{\Gamma})}
 \int_{\widehat{\Gamma}}\gamma(x^{-1}(hs_i)^{-1})
 \  dm_{\widehat{\Gamma}}(\gamma)\right).\nonumber
\end{eqnarray*}
Now, note that for $x = id_{\Gamma}$,  $\mathfrak{F}^{-1}(\widetilde{a})(id_{\Gamma})= k.$ But,
for $x\neq id_{\Gamma}$, from Proposition~\ref{prop:char-orthogonal} we observe that $\mathfrak{F}^{-1}(\widetilde{a})(x)\neq 0$ if and only if $x=(hs_i)^{-1}$  for some $h\in H$ and some $s_i \in S, 1\leq i \leq k $. Thus, using the above argument and the condition $S=S^{-1}$,  we have
$$\mathfrak{F}^{-1}(\widetilde{a})(x)\ =
k \ \chi_{\{id_{\Gamma}\}}(x)- \frac{1}{|H|}\ \sum_{i=1}^k \chi_{s_i H}(x).$$
Since $S=S^{-1},$ it can  easily  be seen  that $\{s_1 H,\ldots,
s_k H\}=\{s_1^{-1}H,\ldots, s_k^{-1}H\}$. Therefore, whenever $x=b h_0\in b H$,
we get
\begin{eqnarray*}
\mathfrak{F}^{-1}(\widetilde{a}*\widetilde{f})(x) \hspace{-.1in}
&=& \hspace{-.1in}\sum_{y\in \Gamma}\mathfrak{F}^{-1}(\widetilde{a})(y)\widetilde{f}(xy^{-1}) 
     = \sum_{y \in \Gamma}\left(k \ \chi_{\{id_{\Gamma}\}}- \frac{1}{|H|}
      \sum_{i=1}^k \chi_{s_i H}\right)(y)\widetilde{f}(xy^{-1}) \nonumber \\
     &=&  k \  \widetilde{f}(x) -\frac{1}{|H|} \ \sum_{i=1}^k
     \sum_{y \in {s_i H}}\widetilde{f}(xy^{-1}) =  k \ \widetilde{f}(x) -\frac{1}{|H|} \ \sum_{i=1}^k
     \sum_{h \in H} \widetilde{f}(b h s_i^{-1}). \label{eqn:convoa1}
\end{eqnarray*}
The above equation implies that
$\mathfrak{F}^{-1}(\widetilde{a}*\widetilde{f})(x)=\mathfrak{F}^{-1}
(\widetilde{a}*\widetilde{f})(z)$
 for all $x,z \in b H$  and hence the required result follows.
\end{proof}

Thus, we have established the initial result   which  helps us in proving
the main result, stated next as Theorem~\ref{thm:heat21},  on the combinatorial
heat equation on infinite coset graphs.

\begin{theorem}\label{thm:heat21}
Let $\Gamma$ be an infinite discrete abelian group,  $H$  a finite subgroup of $\Gamma$ and
$S$   a finite subset of $\Gamma$ such that $S\subset \Gamma \setminus
H$, $S^{-1}= S$ and $H\cup S$ generates $\Gamma$. Let
$G={\mbox{Coset }}(\Gamma,H, S)$ be the coset graph of $\Gamma$ with respect to
$H$ and $S$.  Then,
for any  $f\in L^2(V)$,
 the combinatorial heat equation (\ref{eqn:heat-graph}) on $G$ admits a unique solution if and
only if the initial value problem represented by
(\ref{eqn:changeheat11})  admits a unique solution.

\begin{proof}
By Lemma~\ref{lem:convertheat1},   (\ref{eqn:changeheat11})
 admits a unique solution $\widetilde{u}(x,n)$  with  $\widetilde{u}(x,n)$ being a constant on
 every left coset of $H$. Hence,
(\ref{eqn:changeheat11}) can be rewritten as
\begin{equation}\label{eqn:defu}
\begin{split}
\frac{1}{|H|} \left( k|H| \  \widetilde{u}(x,n)  \ -\ |H|\  \sum_{i=1}^k  \  \widetilde{u}(x h s_i, n)  \right)
         &+ \widetilde{u}(x,n+1)  - \widetilde{u}(x,n)=0, \\
\widetilde{u}(x,0)&=\widetilde{f}(x).
\end{split}
\end{equation}
To proceed further, we define $u:V(G)\times \Z_+\rightarrow
~\C$ by $ u(b H,n)=\widetilde{u}(x,n),$ whenever $x \in bH.$ Then,
 (\ref{eqn:defu}) can  be rewritten in the form
\begin{equation*}\label{eqn:convert1}
k \  u(\cX,n)\ -\ \sum_{i=1}^k  \  u(\cX s_i, n) +  u(\cX,n+1) -u(\cX,n)=0,  \ \mbox{and} \
u(\cX,0)=f(\cX),
\end{equation*}
where $\cX(=bH, {\mbox{say}})$ represents some left coset of $H$ in $\Gamma$, {\it i.e.}
$u(\cX,n)$ is a solution to (\ref{eqn:heat-graph}) on $G$. Since,  $\widetilde{u}(x,n)$ is a unique solution
to (\ref{eqn:changeheat11}), so uniqueness of $u(\cX,n)$ follows.

Similarly, for a  given solution $u(\cX,n)$ to (\ref{eqn:heat-graph}) on $G$, we define
$\widetilde{u}:V(\widetilde{G})\times \Z_+\rightarrow \C$ by $ \widetilde{u}(x,n) = u(bH,n),$
whenever $x \in b H.$ Then, it is  easily seen that
 $ \widetilde{u}(x,n)$ gives a solution to (\ref{eqn:changeheat11}) and the uniqueness
 follows  from Lemma~{\ref{lem:convertheat1}}.
\end{proof}
\end{theorem}

Before proceeding further we first state and prove the following lemma.
\begin{lem}\label{lem:convertwave12}
Let $\Gamma$ be an infinite discrete abelian group,  $H$  a finite subgroup of $\Gamma$ and
$S$   a finite subset of $\Gamma$ such that $S\subset \Gamma \setminus
H$, $S^{-1}= S$ and $H\cup S$ generates $\Gamma$. Let
$G={\mbox{Coset }}(\Gamma,H, S)$ be the coset graph of $\Gamma$ with respect to $H$ and $S$.
Let  $\widetilde{G}$ be the graph as defined in the paragraph leading to Lemma~\ref{lem:convertheat1}
and let $\widetilde{f},\widetilde{g}\in L^1(\Gamma)$. Then, the initial value problem
\begin{equation}\label{eqn:converwave12}
\begin{split}
&\frac{1}{|H|}\Delta_{\widetilde{G}}\widetilde{u}(x,n)+\partial^2
     \widetilde{u}(x,n)=0 \quad on \ V(\widetilde{G})\times \Z_+,\\
&\widetilde{u}(x,0)=\widetilde{f}(x), \
           \partial_n \widetilde{u}(x,0)=\widetilde{g}(x);\\
\end{split}
\end{equation}
admits a unique solution if and only if $\ \widehat{\widetilde{g}}(\gamma_0)=0$,
 where $\gamma_0$ is the trivial character of the group $\Gamma$.
In case the solution exists, one has $\widetilde{u}(x,n)=
{\mbox{constant}}$ on each left coset of $H$.
\end{lem}

\begin{proof}
Let $\{H s_1, H s_2, \ldots, H s_k \}$ be the set of all  distinct
right cosets of $H$ in $\Gamma$, where $s_i \in S,  1\leq i\leq k$. Using
(\ref{eqn:laplacian}) and  applying the Fourier transform on~(\ref{eqn:converwave12})
with respect to group $\Gamma$, we get
\begin{equation*}
\begin{split}
 \widehat{\widetilde{u}}&(\gamma, n+2)- 2\widehat{\widetilde{u}}(\gamma, n+1)-
         \left[-\widetilde{a}(\gamma)-1\right]\widehat{\widetilde{u}}(\gamma,n)=0, \\
 \widehat{\widetilde{u}}&(\gamma,0)=\widehat{\widetilde{f}}(\gamma)  \; {\mbox{ and }} \;
   \widehat{\widetilde{u}}(\gamma,1)-
   \widehat{\widetilde{u}}(\gamma,0)=\widehat{\widetilde{g}}(\gamma),
\end{split}
\end{equation*}
where $\widetilde{a}(\gamma)=k-\frac{1}{|H|}\sum_{h \in H}
\sum_{i=1}^k\gamma((hs_i)^{-1}).$ Using arguments similar to those in
the proof of Theorem~\ref{thm:wave11}, the above equation has a unique
solution if and only if $\ \widehat{\widetilde{g}}(\gamma_0)=0$. In case
the solution exists, it has the form
\begin{equation}\label{eqn:systemeqn}
\left.\begin{array}{l} \hspace*{1in} \widetilde{u}(x,n) = \widetilde{F}_n*\widetilde{f}(x)+ \widetilde{G}_n*\widetilde{g}(x), {\mbox{  where }}  \\
 \widetilde{F}_n(x)
 =  \chi_{\{id_\Gamma\}}(x) \ +\ \sum_{i=1}^{[\frac{n}{2}] } \ (-1)^i \  {n \choose 2i}\ \underbrace{\mathfrak{F}^{-1}(\widetilde{a})*\cdots*
 \mathfrak{F}^{-1}(\widetilde{a})}_{i \ times}(x),\\
\widetilde{G}_n(x)  =  n\chi_{\{id_\Gamma\}}(x) \ +\ \sum_{i=1}^{[\frac{n-1}{2}] } \ (-1)^i \ {n \choose 2i+1} \
\underbrace{\mathfrak{F}^{-1}(\widetilde{a})*\cdots*\mathfrak{F}^{-1}(\widetilde{a})}_{i \ times}(x). \end{array} \right\}
\end{equation}
Now, from Lemma~\ref{lem:convertheat1},   we see that $\mathfrak{F}^{-1}(\widetilde{a})(x)\ =
k \ \chi_{\{id_{\Gamma}\}}(x)- \frac{1}{|H|}\ \sum_{i=1}^k \chi_{s_i H}(x)$ and for
any $x,z \in~b H$, $\mathfrak{F}^{-1}(\widetilde{a}*\widetilde{f})(x)=\mathfrak{F}^{-1}
(\widetilde{a}*\widetilde{f})(z)$ and hence the desired result follows.
\end{proof}

This brings us to the final result of this section which  deals with the
combinatorial wave equation on infinite coset graphs. We omit the proof
as the proof is  similar to the proof of  Theorem~\ref{thm:heat21} and
also  in  view of Lemma~\ref{lem:convertwave12} which forms the initial
step in  the proof.

\begin{theorem}\label{thm:wave21}
Let $\Gamma$ be an infinite discrete abelian group,  $H$ be a finite subgroup of $\Gamma$ and let
$S$ be  a finite subset of $\Gamma$ such that $S\subset \Gamma \setminus
H$, $S^{-1}= S$ and $H\cup S$ generates $\Gamma$. Let
$G={\mbox{Coset }}(\Gamma,H, S)$ be the coset graph of $\Gamma$ with respect to
$H$ and $S$.
Given $f,g\in L^1(V)$ the combinatorial wave equation
\begin{equation*}\label{eqn:wave21}
\begin{split}
&\Delta_G u(\cX,n)+\partial_n^2 u(\cX,n)=0 \quad on \ V(G)\times \Z_+,\\
& u(\cX,0)=f(\cX), \ \partial_n u(\cX,0)=g(\cX);\\
\end{split}
\end{equation*}
has a solution if and only if $\ \widehat{\widetilde{g}}(\gamma_0)=0$,
 where $\gamma_0$ is the trivial character of the group $\Gamma$, and
 $\widetilde{g}:\Gamma \rightarrow \C$ is defined by
$\widetilde{g}(x)\ = \ g(bH), \  {\mbox{ whenever }} \ x \in bH.$
\end{theorem}


\section{Combinatorial Heat and Wave Equations on Regular Trees}\label{sec:reg-tree}
In this Section, we solve the combinatorial heat and wave equations on a
$k$-regular tree, denoted $\tree$, which can be identified with an infinite
Cayley graph $G=\mbox{Cay}(\Gamma,S)$, where $\Gamma$ is a non-abelian free
group with generators $s_1,  \ldots, s_k$, each of order $2$.

A similar problem related to wave equation on $k$-regular
trees have also been  in~\cite{Cohen,Medolla}.
In particular, the authors in \cite{Medolla} looked at the normalized Laplacian operator $\mathcal{L}$ defined by
$$\mathcal{L}f(x)=\frac{1}{k}\sum_{y\sim x}[f(x)-f(y)] \ \mbox{ for all }\  x\in V(\tree)$$
and studied the problem $-(\mathcal{L}-b)u(x,t)=\partial_t^2 u(x,t)  {\mbox{ on }} V(\tree) \times [0,\infty)$ with initial conditions
$u(x,0)=f(x),  \partial_t   u(x,0)=g(x)$,
where $b\geq0$ and $\partial_t$ denotes the partial derivative with respect to the
continuous variable $t$. In~\cite{Cohen}, the authors studied a similar problem
with discrete time variable. They considered  the problem
$-2 \mathcal{L}u(x,n)=u(x,n+1)+u(x,n-1)-2u(x,n)  {\mbox{ on }} V(\tree) \times [0,\infty)$ with initial conditions $u(x,0)=f(x),    u(x,1)=g(x).$

Our approach to  (\ref{eqn:heat-graph}) and (\ref{eqn:wave-graph})
is analogous to the approach in the classical case of solving the wave equation
on Euclidean space $\R^n$. The structure of $\tree$ and the techniques we have
used in Section~$2$, allow us to adopt such an approach. For details of the
classical approach, see~\cite[pages~$166-169$]{Folland}.

\subsection{ Notations and Preliminary  Results}
\label{sec:Prelims On T}

Let  $\path = (V(\path),E(\path))$ be the infinite path with $V(\path) = \Z$ and  $E(\path) = \{ \{r,r+1\}: r \in \Z\}$.
Thus, for $k=2$ the tree $\tree$ is same as the path $\path$. Also, note that  $\tree$ is a metric space with respect to the metric $d_{\tree}$, in short $d$. If  $Aut(\tree)$ is the automorphism group of  $\tree$ then  $Aut(\tree)$  corresponds to the group of isometries of the metric space $\tree$. Also, the symmetry of $\tree$ implies that for $x,y \in V(\tree)$, there exist a $\varphi \in Aut(\tree)$ such that $\varphi(x)=y$, {\it i.e.},
$\tree$ is a vertex-transitive graph.

Now, for a fixed  vertex $x \in V(\tree)$ and  $r \in \Z_+$, let $S(x,r) = \{z \in V(\tree): d(z,x)=r\}$ be the boundary of the ball of radius $r$ centered
at  $x$. Then, symmetry of  $\tree$ implies that $|S(x,r)|$ is independent of  $x$ and hence, we write $ |S(x,r)| = S(r)$. Thus, for  $x\in V(\tree)$,
$$ S(r) = |S(x,r)|= \Big|\{z \in V(\tree): d(z,x)=r\}\Big|=\left\{\begin{array}{ll}
                                        1, &  r= 0, \\
                                        k(k-1)^{r-1}, & r>0.
                                    \end{array}\right.$$
Also, recall that for a fixed $r \in \Z_+$ and
$\varphi:V(\tree)\rightarrow \C$, the spherical mean  of $\varphi$ is defined as\label{page:sp-mean}
$$ M_{\varphi}(x,r)=\frac{1}{S(r)}\sum_{d(z,x)=r}\varphi(z) \; {\mbox{ for all }} \; x \in V(\T).$$
Now,  for each $x \in V(\tree)$, we extend the definition of  $M_{\varphi}(x,r)$, as an even function, to all $r\in \Z$. That is, for $\varphi:V(\tree)\rightarrow \C$, we have a map
$M_\varphi: V(\tree)\times V(\path)\rightarrow \C$ defined by
$M_{\varphi}(x,r) = M_{\varphi}(x,- r)$ for all $r \in \Z$.
With the notations and definitions as above, we state our next  result.

\begin{lem}\label{lem:mu-l1}
For $1\leq p <\infty$, if $\varphi \in L^p(V(\tree))$ then $M_{\varphi}(x,\cdot) \in  L^p(\Z)$ for all $x \in V(\tree)$.
\end{lem}
\begin{proof}
Fix $x \in V(\tree)$ and let $1< p<\infty$ with $\frac{1}{p}+\frac{1}{p'}=1$. Then,
by the definition of $ M_{\varphi}$
\begin{eqnarray}\label{eqn:lp-inq}
\sum_{r\in \Z}|M_{\varphi}(x,r)|^p
\leq  2\sum_{r=0}^{\infty}|M_{\varphi}(x,r)|^p
              =  2\sum_{r=0}^{\infty}\left|\frac{1}{S(r)}
                     \sum_{d(z,x)=r}\varphi(z)\right|^p.
\end{eqnarray}
Now, by Holder's inequality, we have
\begin{eqnarray*}
\sum_{d(x,z)=r}|\varphi(z)|
              \leq \biggl(\sum_{d(x,z)=r}1\biggr)^{1/p'}
                     \biggl(\sum_{d(x,z)=r}|\varphi(z)|^p\biggr)^{1/p} 
              =(S(r))^{1/p'} \biggl(\sum_{d(x,z)=r}|\varphi(z)|^p\biggr)^{1/p}.
\end{eqnarray*}
Thus, from (\ref{eqn:lp-inq}), we obtain \vspace{-.2cm}
\begin{eqnarray*}
\sum_{r\in \Z}|M_{\varphi}(x,r)|^p
              &\leq& 2\sum_{r=0}^{\infty} \frac{1}{(S(r))^p}(S(r))^{\frac{p}{p'}}
                         \sum_{d(x,z)=r}|\varphi(z)|^p
              =2\sum_{r=0}^{\infty} (S(r))^{-p\left(1-\frac{1}{p'}\right)}
                           \sum_{d(x,z)=r}|\varphi(z)|^p\\
              &=&2\sum_{r=0}^{\infty} (S(r))^{-1} \sum_{d(x,z)=r}|\varphi(z)|^p
              \leq 2\sum_{r=0}^{\infty}
                           \sum_{d(x,z)=r}|\varphi(z)|^p
                           = 2\sum_{z \in V(\tree)}|\varphi(z)|^p < \infty.
\end{eqnarray*}
Similarly, for $p=1$, we  have
\begin{eqnarray*}
\sum_{r\in \Z}|M_{\varphi}(x,r)|
              \leq  2\sum_{r=0}^{\infty}\frac{1}{S(r)}
                     \sum_{d(x,z)=r}|\varphi(z)|
                     \leq  2\sum_{r=0}^{\infty} \sum_{d(z,x)=r}|\varphi(z)|
              = 2\sum_{z \in V(\tree)}|\varphi(z)| < \infty.
\end{eqnarray*}
Hence the desired result follows.
\end{proof}

To proceeding further, we extend the definition of the difference operator $\partial_n$ (see~Definition~\ref{def:diff-opr}) to an operator  $\partial_n : \Z \rightarrow \C$ by
$\partial_n v(n)=v(n+1)-v(n)$ for all $n\in \Z$.
  The next result
establishes an important relation between $\Delta_\tree$ and $\Delta_\path$. This result can be viewed as  an analogue
of ``Darboux equation" (for the classical case, see~\cite[ Theorem $6.2.1$]{jost}).

\begin{theorem}\label{thm:lap-convert}
Let $\tree$ and  $\path$ be defined as above.
If  $\varphi$ is a complex valued function on $V(\tree)$ then the spherical
mean $M_{\varphi}$ satisfies the following equation
$$(\Delta_\tree M_{\varphi})(x,r)= ((\Delta_\path + (2-k)\partial_r) M_{\varphi}) (x,r) {\mbox{ for all }} x \in V(\tree), r \in \Z_+.$$
\end{theorem}

\begin{proof}
Fix $x\in V(\tree)$ and $r \in \Z_+$. Then,  using (\ref{eqn:laplacian})
and the definition of $M_{\varphi}$, we have
\begin{eqnarray}\label{eqn:lap-TP}
(\Delta_\tree M_{\varphi})(x,r)
= k \;M_{\varphi}(x,r)- \sum_{y\sim x}
\biggl(\frac{1}{|S(r)|}\sum_{d(z,y)=r}\varphi(z)\biggr).
\end{eqnarray}
To compute the second term of (\ref{eqn:lap-TP}), we observe the following.
Let $y_1, \ldots, y_k$ be the vertices of $\tree$ that are adjacent to $x$.
Then, for each fixed $i, 1 \le i \le k$, a vertex $z \in S(y_i, r)$, if the path
from $y_i$ to $z$ passes through the vertex $x$ and $z \in S(x,r-1)$ or the path
from $y_i$ to $z$ does not pass through $x$ and $z \in S(x,r+1)$, {\it i.e.}, for
each fixed $i, 1 \le i \le k$, one has
$$S(y_i, r) = \left\{S(x,r-1) \setminus S(y_i,r-2)\right\} \cup
   \left\{ S(x,r+1) \setminus \left(\cup_{j\ne i} S(y_j, r) \right)\right\}.$$
Thus, as we sum over all $y \sim x$ in the second term of (\ref{eqn:lap-TP}),
we get exactly $(k-1)$ copies of $S(x,r-1)$ and one copy of $S(x,r+1)$. Hence,
with these observations  (\ref{eqn:lap-TP}) reduces to
\begin{eqnarray}
\Delta_\tree(M_{\varphi}(x,r))
               &=& k\;M_{\varphi}(x,r)- \frac{1}{S(r)}
                    \left[\sum_{d(z,x)=r+1}\varphi(z)
                      +(k-1)\sum_{d(z,x)=r-1}\varphi(z) \right] \nonumber\\
               &=&  k\;M_{\varphi}(x,r)-\frac{1}{S(r)}\Big[ S(r+1)\;M_{\varphi}(x,r+1)
                      + (k-1)\;S(r-1)\;M_{\varphi}(x,r-1)\Big] \nonumber\\
               &=&  k\;M_{\varphi}(x,r) -(k-1)\;M_{\varphi}(x,r+1) -  M_{\varphi}(x,r-1) \label{eqn:lap2}\\
               &=& \Big[2 M_{\varphi}(x,r)- M_{\varphi}(x,r+1) - M_{\varphi}(x,r-1)\Big]
                   +(k-2)\Big[ M_{\varphi}(x,r)-M_{\varphi}(x,r+1)\Big] \nonumber\\
               &=& \Delta_P(M_{\varphi}(x,r))  -  (k-2)\partial_r (M_{\varphi}(x,r))\nonumber\\
               &=& \Big(\Delta_P + (2-k)\partial_r\Big)(M_{\varphi}(x,r))\nonumber.
\end{eqnarray}
Hence,  we have obtained the desired result.
\end{proof}

\begin{cor}\label{cor:wave-convert}
Suppose $u(x,n)$ is a function on $V(\tree)\times \Z_+$ and $M_u(x,r,n)$ denotes the
spherical mean of the function $x\mapsto u(x,n)$. Then, for all $x\in V(\tree)$, $u$ satisfies
\begin{enumerate}[label={(\arabic*)}]
 \item  the combinatorial  heat equation (\ref{eqn:heat-graph}) on $\tree$ if and only if $M_u$ satisfies
\begin{equation}\label{eqn:heat-convert}
    \Big(\Delta_\path + (2-k)\partial_r\Big)(M_u(x,r,n))+ \partial_{n} M_u(x,r,n)=0,
\end{equation}

\item  the combinatorial wave equation (\ref{eqn:wave-graph}) on $\tree$ if and only if $M_u$ satisfies
\begin{equation}\label{eqn:wave-convert}
    \Big(\Delta_\path + (2-k)\partial_r\Big)(M_u(x,r,n))+ \partial_{n}^2 M_u(x,r,n)=0.
\end{equation}

\end{enumerate}
\end{cor}

\begin{proof}
For any $x\in V(\tree)$ and $r,n \in \Z_+$, note that  using (\ref{eqn:laplacian})
and the definition of $M_{\varphi}$, we have
\begin{eqnarray*}
M_{(\Delta_Tu)}(x,r,n)&=&\frac{1}{S(r)}\sum_{d(z,x)=r}\Delta_T u(z,n)
                    = \frac{1}{S(r)} \sum_{d(z,x)=r}
                         \left(k \;u(z,n)-\sum_{y\sim z}u(y,n) \right)\nonumber\\
                    &=&  \frac{k}{S(r)}\sum_{d(z,x)=r} u(z,n)-
                         \frac{1}{S(r)} \sum_{d(z,x)=r}\sum_{y\sim z}u(y,n)\nonumber.
\end{eqnarray*}
Now an argument that is similar to the argument in the proof of
Theorem~\ref{thm:lap-convert} implies that as we sum over all $y \sim z$
as $z$ varies over $S(x,r)$, we get exactly $(k-1)$ copies of $S(x,r-1)$
and one copy of $S(x,r+1)$. Therefore, the above equation can be rewritten as
\begin{eqnarray}
M_{(\Delta_\tree u)}(x,r,n)&=& k M_u(x,r,n)-\frac{1}{S(r)}\left[\sum_{d(y,x)=r+1}u(y,n)
                               + (k-1) \sum_{d(y,x)=r-1}u(y,n)\right]  \nonumber\\
                           &=&  k\; M_u(x,r,n)- (k-1) M_u(x,r+1,n) - \;M_u(x,r-1,n) \nonumber.
\end{eqnarray}
Thus, using (\ref{eqn:lap2}) and the above, we obtain
$\Delta_\tree(M_u(x,r,n))=M_{(\Delta_\tree u)}(x,r,n)$. On similar lines, it
can be verified that $\partial_{n} M_u(x,r,n)=M_{(\partial_{n} u)}(x,r,n)$ and
$\partial_{n}^2 M_u(x,r,n)=M_{(\partial_{n}^2 u)}(x,r,n)$.
Hence, using Theorem~\ref{thm:lap-convert}, the desired results follow.
\end{proof}

In the next section, we
proceed to solve the combinatorial heat and  wave equations on $\tree$ which as stated
earlier corresponds to  a Cayley graph of an infinite non-abelian free group.

\subsection{Results on Combinatorial Heat and Wave Equations on $\tree$}\label{sec:wave-tree}
We are now interested in solving the combinatorial heat and wave equations
on $\tree$. 
In view of
Corollary~\ref{cor:wave-convert}, we observe that if  $u$ is a solution
to (\ref{eqn:heat-graph}) on $\tree$ then the spherical mean $M_u$ satisfies (\ref{eqn:heat-convert})
with the initial condition $M_u(x,r,0)=M_f(x,r)$. Similarly,  if  $u$ is a solution
to (\ref{eqn:wave-graph})  on $\tree$ then the spherical mean $M_u$ satisfies (\ref{eqn:wave-convert})
with the initial conditions
$M_u(x,r,0)=M_f(x,r), \ \partial_n M_u(x,r,0)= M_g(x,r).$
To simplify our notations, for any fixed $x\in V(\tree)$,  let
\begin{equation}\label{eqn:convert-sol}
\widetilde{u}(r,n)=M_u(x,r,n), \  {\mbox{ for }} r \in \Z, n \in \Z_+.
\end{equation}
Accordingly, we  re-write the initial conditions as
\begin{equation*}\label{eqn:initial-convert}
\widetilde{f}(r)=M_f(x,r)=M_u(x,r,0)= \widetilde{u}(r,0) \ \mbox{and} \
\widetilde{g}(r)=M_g(x,r)=M_{\partial_n u}(x,r,0)= \partial_n \widetilde{u}(r,0).
\end{equation*}

Then, it can easily be observed that  solving the combinatorial heat equation (\ref{eqn:heat-graph})
on $\tree$ is equivalent to solving the initial value problem (IVP) on the  infinite path $\path$
\begin{equation}\label{eqn:convert-heat11}
\begin{split}
(\Delta_\path + (2-k)\partial_r)\widetilde{u}(r,n)+\partial_{n} \widetilde{u}(r,n)
                                          &=0  {\mbox{ on }} V(\path) \times \Z_+,\\
\widetilde{u}(r,0)&=\widetilde{f}(r),
\end{split}
\end{equation}
and  solving the combinatorial wave equation (\ref{eqn:wave-graph})
on $\tree$ is equivalent to solving the IVP
\begin{equation}\label{eqn:convert-wave1}
\begin{split}
(\Delta_\path + (2-k)\partial_r)\widetilde{u}(r,n)+\partial_{n}^2 \widetilde{u}(r,n)                                                              &=0 \ \mbox{on} \ V(\path)\times \Z_+,  \\
\widetilde{u}(r,0)=\widetilde{f}(r),  \  \  \partial_n \widetilde{u}(r,0)&=\widetilde{g}(r).
\end{split}
\end{equation}

We now state a result that
gives information about the solution of (\ref{eqn:convert-heat11}).
\begin{lem}\label{lem:sol-heat-convert-u}
Let $f\in L^2(V(\tree))$ and let  (\ref{eqn:convert-heat11}) admit a
solution $\widetilde{u}(r,n)$. Then $\widetilde{u}(\cdot,n) \in L^2(V(\path))$ for all $n\in \Z_+$.
\end{lem}
\begin{proof}
The result is clearly true for $n=0$. Now, using  (\ref{eqn:lap2}), we can re-write (\ref{eqn:convert-heat11}) as
\begin{equation*}\label{eqn:rec-heat:2}
 \widetilde{u}(r, n+1)= (k-1)\widetilde{u}(r+1,n)+\widetilde{u}(r-1,n)-(k-1)\widetilde{u}(r,n) {\mbox{ and}} \;\;
 \widetilde{u}(r,0)=\widetilde{f}(r).
\end{equation*}
 Since $f\in L^2(V(\tree))$ by Lemma~\ref{lem:mu-l1},
 $\widetilde{f}\in L^2{(V(\path))}$.
Therefore, by mathematical induction on  $n$,
 the desired result follows.
\end{proof}
Now we are ready to solve the initial value problem (\ref{eqn:convert-heat11}).
\begin{lem}\label{lem:heat-con-tree}
Let $f\in L^2(V(\tree))$.  Then, the initial value problem (\ref{eqn:convert-heat11})
admits a unique solution
$\widetilde{u}(r,n)=K_n*\widetilde{f}(r),$
where $K_n(r)= \mathfrak{F^{-1}}\left( \sum_{j=0}^{n} {n \choose j}  (-a)^j\right)(r)$
with $a(t)=-(k-1)e^{-it}+k- e^{it}$.
\end{lem}
\begin{proof}
Using  (\ref{eqn:lap2}), we can re-write (\ref{eqn:convert-heat11}) as
\begin{equation*}
 k \widetilde{u}(r,n)-(k-1)\widetilde{u}(r+1,n)- \widetilde{u}(r-1,n)
  +  \widetilde{u}(r,n+1)- \widetilde{u}(r,n)=0 \ \mbox{with} \
 \widetilde{u}(r,0)=\widetilde{f}(r).
\end{equation*}
Since $\widetilde{f}\in L^2(V(\path))$,  by Lemma~\ref{lem:sol-heat-convert-u},
$\widetilde{u}(\cdot,n) \in L^2(V(\path))$ for all $n\in \Z_+$. Thus, applying Fourier transform  with respect to  $r\in \Z$  and
proceeding similar to the proof of Theorem~\ref{thm:heat11}, we get
$$\widehat{\widetilde{u}}(\cdot,n)=[1-a(\cdot)]^n \widehat{\widetilde{f}}(\cdot), \
\mbox{where} \ a(t)=-(k-1)e^{-it}+k- e^{it}. $$
Using the inverse Fourier transform, we obtain the desired result.
\end{proof}

Now we compute the trigonometric polynomials $a^j(t)$, for  $j\geq 0$.
\begin{eqnarray*}
a^j(t) &=& \big[-(k-1)e^{-it}+k- e^{it}\big]^j=  (1-e^{-it})^j \;[(k-1)-e^{it}]^j\\
       &=& \left[ \sum_{m=0}^j{j \choose m}(-1)^{j-m}\;  e^{-i(j-m)t}\right]
          \left[ \sum_{l=0}^j{j \choose m}(-1)^{j-m}\;(k-1)^l\; e^{i(j-l)t}\right]
       = \sum_{s=-j}^j \alpha_{s}^j \; e^{ist},
\end{eqnarray*}
where
\begin{equation}\label{eqn:coef-aj}
\alpha_{s}^j= \left\{\begin{array}{ll}
                   (-1)^s \ds\sum_{l=0}^{j-s}{j \choose l+s} {j \choose l}(k-1)^l,          & \mbox{if}\ s\geq 0, \\
                   (-1)^{-s} \ds\sum_{l=0}^{j+s}{j \choose l-s} {j \choose l}(k-1)^{l-s },   &  \mbox{if}\ s<0.
               \end{array}\right.
\end{equation}

Now we state and prove the combinatorial heat equation on $\tree$.
\begin{theorem}\label{thm:tree-heat-final}
Let $\mathbf T$ be a $k$-regular tree and let us denote $\beta_s=\dfrac{1+(k-1)^s}{k(k-1)^{s-1}}$.
If   $f\in L^2(V(\mathbf T))$ then
the combinatorial heat equation (\ref{eqn:heat-graph}) on $\tree$
has a  unique solution
$\ds {u(x,n)=\sum_{s=0}^{n}\biggl(w_s^{(n)}\sum_{d(x,y)=s}f(y) \biggr)}$, where
\begin{equation*}
w_s^{(n)}= \left\{\begin{array}{ll}{\ds \beta_s
\sum_{l \ge 0} \biggl[ \sum_{j\geq l} (-1)^j {n \choose j+s} {j+s \choose l}{j+s \choose l+s} \biggr](k-1)^l}, & \mbox{ if }\ s>0, \\
             {\ds  \sum_{l\ge 1} \biggl[\sum_{j\geq l}(-1)^j {n \choose j}{j \choose l}^2 \biggr](k-1)^l}, &  \mbox{ if }\  s=0.\\
                   \end{array}\right.
\end{equation*}
\end{theorem}


\begin{proof}
Using (\ref{eqn:convert-sol}), for a fixed $x\in V(\tree)$, we obtain
$\widetilde{u}(0,n)=M_u(x,0,n)=u(x,n)$. In view of
Corollary~\ref{cor:wave-convert} and Lemma~\ref{lem:heat-con-tree}, the
solution to the combinatorial heat equation on $\tree$ is
\begin{equation*}
u(x,n)=\widetilde{u}(0,n)=K_n*\widetilde{f}(0).
\end{equation*}
Since $\widetilde{f}(s)=\widetilde{f}(-s)$ for all $s\in \Z$,  we have
\begin{eqnarray*}
K_n*\widetilde{f}(0) =  \sum_{s\in \Z} \widetilde{f}(s)K_n(-s)
       =  K_n(0)\widetilde{f}(0)+ \sum_{s=1}^\infty[K_n(s)+ K_n(-s)]  \widetilde{f}(s).
\end{eqnarray*}
Further, if $a^j(t)=\ds \sum_{s=-j}^j \alpha_{s}^j \; e^{ist}$ then $ K_n(s)=\ds\sum_{j=0}^{n}{n \choose j}(-1)^j \mathfrak{F}^{-1}(a^j)(s)$.
Thus,  using Remark~\ref{rem:four1}, we get
\begin{equation*}
 K_n(s)=\left\{\begin{array}{ll}\ds\sum_{j\geq s}^{n}{n \choose j}(-1)^j \alpha_s^j
 , & \mbox{ if }\ -n\leq s\leq n, \\
              {\ds 0}, &  \mbox{ otherwise }.\\
               \end{array}\right.
\end{equation*}
Therefore, by (\ref{eqn:coef-aj}),
for $0<s\leq n$, we have
\begin{eqnarray*}
K_n(s)+K_n(-s)\!\! \! &=&\!\! \!  \sum_{j\geq s}^{n}{n \choose j}(-1)^j (\alpha_s^j+\alpha_{-s}^j)\\
              &=& \!\! \! \sum_{j\geq s}^{n}{n \choose j}(-1)^{j+s}
                      \sum_{l=0}^{j-s}{j \choose l+s} {j \choose l}(k-1)^l[1+(k-1)^s]\\
              &=& \!\! \! [1+(k-1)^s]  \sum_{l \ge 0}
                     \biggl[ \sum_{j\geq l} (-1)^j {n \choose j+s} {j+s \choose l}{j+s \choose l+s}
                            \biggr](k-1)^l,            
\end{eqnarray*}
and for $s=0$ with the convention  ${0 \choose 0}=1$ , we have
\begin{eqnarray*}
K_n(0)\!\! \!  &=& \!\! \!  \sum_{j= 0}^{n}{n \choose j}(-1)^j \alpha_0^j
              =  \sum_{j\geq 0}^{n}{n \choose j}(-1)^j \sum_{l=0}^j
                              {j \choose l}^2 (k-1)^l=\sum_{l\ge 0} \biggl[\sum_{j\geq l}(-1)^j {n \choose j}{j \choose l}^2 \biggr](k-1)^l\\
              &=& \!\! \!  \sum_{j\geq 0}(-1)^j {n \choose j} + \sum_{l\ge 1} \biggl[\sum_{j\geq l}(-1)^j {n \choose j}{j \choose l}^2 \biggr](k-1)^l
               =\sum_{l\ge 1} \biggl[\sum_{j\geq l}(-1)^j {n \choose j}{j \choose l}^2 \biggr](k-1)^l
\end{eqnarray*}

Therefore,
\begin{eqnarray*}
 K_n\hspace{-.15in}&*&\hspace{-.15in}\widetilde{f}(0) \\
&=& \hspace{-.02in} \biggl[ \sum_{l\ge 1} \biggl[\sum_{j\geq l}(-1)^j {n \choose j}{j \choose l}^2 \biggr](k-1)^l\biggr]\widetilde{f}(0)\\
     && \hspace{-.02in} +\ \sum_{s=1}^n \biggl[[1+(k-1)^s]  \sum_{l \ge 0}
                     \biggl[ \sum_{j\geq l} (-1)^j {n \choose j+s} {j+s \choose l}{j+s \choose l+s}
                            \biggr](k-1)^l\biggr]\widetilde{f}(s)\\
       &=& \hspace{-.02in} \biggl[ \sum_{l\ge 1} \biggl[\sum_{j\geq l}(-1)^j {n \choose j}{j \choose l}^2 \biggr](k-1)^l\biggr]f(x)\\
       && \hspace{-.02in} + \sum_{s=1}^n \biggl[[1+(k-1)^s]  \sum_{l \ge 0}
                     \biggl[ \sum_{j\geq l} (-1)^j {n \choose j+s} {j+s \choose l}{j+s \choose l+s}
                            \biggr](k-1)^l\biggr] \!\!\cdot \! \biggl[\frac{1}{S(s)} \sum_{d(x,y)=s}f(y)\biggr] \\
        &=& \hspace{-.02in} \sum_{s=0}^{n}\biggl(w_s^{(n)}\sum_{d(x,y)=s}f(y) \biggr).
\end{eqnarray*}
Hence the desired result follows.
\end{proof}

Before proceeding further, we state and prove a property of the solution of (\ref{eqn:convert-wave1}) that will be useful for solving the combinatorial wave equation on $\tree$.

\begin{lem}\label{lem:sol-wave-convert-u}
Let $f, g\in L^1(V(\tree))$ and let  (\ref{eqn:convert-wave1}) admit a
solution $\widetilde{u}(r,n)$. Then $\widetilde{u}(\cdot,n) \in L^1(V(\path))$ for all $n\in \Z_+$.
\end{lem}
\begin{proof}
Using  (\ref{eqn:lap2}), we can re-write (\ref{eqn:convert-wave1}) as
\begin{equation}\label{eqn:rec-wave:2}
\begin{split}
 \widetilde{u}&(r, n+2)= 2\widetilde{u}(r, n+1)-(k+1)\widetilde{u}(r,n)
  +(k-1) \widetilde{u}(r+1, n)+\widetilde{u}(r-1,n), \\
 \widetilde{u}&(r,0)=\widetilde{f}(r)  \; {\mbox{ and }} \;
 \widetilde{u}(r,1)=\widetilde{u}(r,0)+\widetilde{g}(r).
\end{split}
\end{equation}
 As $f, g\in L^1(V(\tree))$, using Lemma~\ref{lem:mu-l1},  we have $\widetilde{f}, \widetilde{g}\in L^1{(V(\path))}$.
Therefore, by applying  mathematical induction on discrete time variable $n$,
 the desired result follows.
\end{proof}

We  now solve the initial value problem (\ref{eqn:convert-wave1}) and
also obtain a necessary and sufficient condition for the existence of the solution.
\begin{lem}\label{lem:wave-con-tree}
Let $f,g\in L^1(V(\tree))$.  Then, the initial value problem (\ref{eqn:convert-wave1})
admits a unique solution if and only if  $\widehat{\widetilde{g}}(0)=0$. In case the
solution exists, it is unique and is expressed by
$$\widetilde{u}(r,n)=F_n*\widetilde{f}(r)+G_n*\widetilde{g}(r),$$
where $F_n(r)= \mathfrak{F^{-1}}\biggl( \sum_{j=0}^{[\frac{n}{2}]} {n \choose 2j}  (-a)^j\biggr)(r)$
and $G_n(r)= \mathfrak{F^{-1}}\biggl( \sum_{j=0}^{[\frac{n-1}{2}]}{n \choose 2j+1}  (-a)^j\biggr)(r)$
 with $a(t)=-(k-1)e^{-it}+k- e^{it}$.
\end{lem}

\begin{proof}
Using  (\ref{eqn:lap2}), we can re-write (\ref{eqn:convert-wave1}) as \begin{equation*}
\begin{split}
 k \widetilde{u}&(r,n)-(k-1)\widetilde{u}(r+1,n)- \widetilde{u}(r-1,n)
  + \widetilde{u}(r,n+2)-2 \widetilde{u}(r,n+1)+ \widetilde{u}(r,n)=0,\\
 \widetilde{u}&(r,0)=\widetilde{f}(r)  \; {\mbox{ and }} \;
                                    \widetilde{u}(r,1)-\widetilde{u}(r,0)=\widetilde{g}(r).
\end{split}
\end{equation*}
Since $\widetilde{f}, \widetilde{g}\in L^1(V(\path))$,  by Lemma~\ref{lem:sol-wave-convert-u},
$\widetilde{u}(\cdot,n) \in L^1(V(\path))$ for all $n\in \Z_+$. Thus, taking the
Fourier transform on both sides with respect to the variable $r\in \Z=V(\path)$, we get
\begin{equation*}
\begin{split}
 k \widehat{\widetilde{u}}&(t,n)-(k-1)e^{-it}\widehat{\widetilde{u}}(t,n)- e^{it}\widehat{\widetilde{u}}(t,n)
     + \widehat{\widetilde{u}}(t,n+2)-2 \widehat{\widetilde{u}}(t,n+1)+ \widehat{\widetilde{u}}(t,n)=0,\\
 \widehat{\widetilde{u}}&(t,0)=\widehat{\widetilde{f}}(t)  \; {\mbox{ and }} \;
                  \widehat{ \widetilde{u}}(t,1)-\widehat{\widetilde{u}}(t,0)=\widehat{\widetilde{g}}(t).
\end{split}
\end{equation*}
By Remark~\ref{rem:hat-f-cont},  $\widehat{\widetilde{u}}(t, n)$ for all $n\in \Z_+$, is a continuous function  on the unit circle $\T$ with respect to the variable $t$. Hence, the point-wise calculation of the above equations
is well defined. Now, re-writing the above equations, we have
\begin{equation*}
\begin{split}
\widehat{\widetilde{u}}&(t,n+2)-2\widehat{\widetilde{u}}(t,n+1)
  - \big[(-a)(t)-1\big]\widehat{\widetilde{u}}(t,n)=0,\\
 \widehat{\widetilde{u}}&(t,0)=\widehat{\widetilde{f}}(t)  \; {\mbox{ and }} \;  \widehat{ \widetilde{u}}(t,1)-\widehat{\widetilde{u}}(t,0)=\widehat{\widetilde{g}}(t),
\end{split}
\end{equation*}
where $a(t)=-(k-1)e^{-it}+k- e^{it}$. Note that, the above recurrence
relation (with respect to the variable $n\in \Z_+$) is  of a form
that is same as the recurrence relation  that appeared in Theorem~\ref{thm:wave11}. Since,  $\Z$ is  also a finitely
generated discrete abelian group  thus, proceeding similar to the proof of Theorem~\ref{thm:wave11}, we obtain the
desired result.
\end{proof}

Now we state and prove the combinatorial wave equation on $\tree$. The ideas and calculations of the proof of the next theorem are
similar to the proof of  Theorem~\ref{thm:tree-heat-final}, but we provide the same, for the sake of completion.
\begin{theorem}\label{thm:tree-wave-final}
Let $\mathbf T$ be a $k$-regular tree.  Then, for $f,g\in L^1(V(\mathbf T))$,
the combinatorial wave equation (\ref{eqn:wave-graph}) on $\tree$
has a  solution if and only if $\widehat{\widetilde{g}}(0)=0$.
Moreover, the solution is unique and if we denote  $\beta_s=\dfrac{1+(k-1)^s}{k(k-1)^{s-1}}$, then
$$u(x,n)=\sum_{s=0}^{[\frac{n}{2}]}\biggl(w_s^{(1,n)}\sum_{d(x,y)=s}f(y) \biggr)+
\sum_{s=0}^{[\frac{n-1}{2}]}\biggl(w_s^{(2,n)}\sum_{d(x,y)=s}g(y)\biggr),$$
where
\begin{equation*}
w_s^{(1,n)}= \left\{\begin{array}{ll}{\ds \beta_s
                     \sum_{l \ge 0}
                     \biggl[ \sum_{j\geq l} (-1)^j {n \choose 2(j+s)} {j+s \choose l}{j+s \choose l+s}
                            \biggr](k-1)^l}, & \mbox{ if }\ s>0, \\
             {\ds  \sum_{j\geq 0}(-1)^j {n \choose 2j} + \sum_{l\ge 1} \biggl[\sum_{j\geq l}(-1)^j {n \choose 2j}{j \choose l}^2 \biggr](k-1)^l}, &  \mbox{ if }\  s=0,\\
                   \end{array}\right.
\end{equation*}
and
\begin{equation*}
w_s^{(2,n)}= \left\{\begin{array}{ll}{\ds \beta_s
\sum_{l \ge 0}
                     \biggl[ \sum_{j\geq l} (-1)^j {n \choose 2(j+s)+1} {j+s \choose l}{j+s \choose l+s}
                            \biggr](k-1)^l}, & \mbox{ if }\ s>0, \\
              {\ds \sum_{j\geq 0}(-1)^j {n \choose 2j+1} + \sum_{l\ge 1} \biggl[\sum_{j\geq l}(-1)^j {n \choose 2j+1}{j \choose l}^2 \biggr](k-1)^l}, &  \mbox{ if }\  s=0.\\
               \end{array}\right.
\end{equation*}
\end{theorem}

\begin{proof}
Using (\ref{eqn:convert-sol}), for a fixed $x\in V(\tree)$, we obtain
$\widetilde{u}(0,n)=M_u(x,0,n)=u(x,n)$. In view of
Corollary~\ref{cor:wave-convert} and Lemma~\ref{lem:wave-con-tree}, the
solution to the combinatorial wave equation on $\tree$ is given by
\begin{equation*}
u(x,n)=\widetilde{u}(0,n)=F_n*\widetilde{f}(0)+G_n*\widetilde{g}(0).
\end{equation*}
Since $\widetilde{f}(s)=\widetilde{f}(-s)$ for all $s\in \Z$,  we have
\begin{eqnarray}\label{eqn:eq-fn}
F_n*\widetilde{f}(0) =  \sum_{s\in \Z} \widetilde{f}(s)F_n(-s)
       =  F_n(0)\widetilde{f}(0)+ \sum_{s=1}^\infty[F_n(s)+F_n(-s)]  \widetilde{f}(s).
\end{eqnarray}
Further, if $a^j(t)=\!\ds\sum_{s=-j}^j \alpha_{s}^j \; e^{ist}$ then
$ F_n(s)=\!\ds\sum_{j=0}^{[\frac{n}{2}]}{n \choose 2j}(-1)^j \mathfrak{F}^{-1}(a^j)(s)$.
Thus,  using Remark~\ref{rem:four1}, we get
\begin{equation*}
F_n(s)= \left\{\begin{array}{ll}{\ds \sum_{j\geq s}^{[\frac{n}{2}]}{n \choose 2j}(-1)^j \alpha_s^j}
 , & \mbox{ if }\ -[\frac{n}{2}]\leq s\leq [\frac{n}{2}], \\
              {\ds 0}, &  \mbox{ if }\  s=0.\\
               \end{array}\right.
\end{equation*}

Therefore, for $s>0$, using (\ref{eqn:coef-aj})  we have
\begin{eqnarray*}
F_n(s)+F_n(-s)&=& \sum_{j\geq s}^{[\frac{n}{2}]}{n \choose 2j}(-1)^j (\alpha_s^j+\alpha_{-s}^j)\\
              &=& \sum_{j\geq s}^{[\frac{n}{2}]}{n \choose 2j}(-1)^{j+s}
                      \sum_{l=0}^{j-s}{j \choose l+s} {j \choose l}(k-1)^l[1+(k-1)^s]\\
               &=& [1+(k-1)^s]  \sum_{l \ge 0}
                     \biggl[ \sum_{j\geq l} (-1)^j {n \choose 2(j+s)} {j+s \choose l}{j+s \choose l+s}
                            \biggr](k-1)^l,
\end{eqnarray*}
and for $s=0$ with the convention  ${0 \choose 0}=1$, we have
\begin{eqnarray*}
F_n(0)  &=& \sum_{j= 0}^{[\frac{n}{2}]}{n \choose 2j}(-1)^j \alpha_0^j
              =  \sum_{j\geq 0}^{[\frac{n}{2}]}{n \choose 2j}(-1)^j \sum_{l=0}^j
                              {j \choose l}^2 (k-1)^l\\
              &=& \sum_{j\geq 0}(-1)^j {n \choose 2j}+\sum_{l\ge 1} \biggl[\sum_{j\geq l}(-1)^j {n \choose 2j}{j \choose l}^2 \biggr](k-1)^l.
\end{eqnarray*}
Thus,  (\ref{eqn:eq-fn}) reduces to
$ \ds F_n*\widetilde{f}(0)=\sum_{s=0}^{[\frac{n}{2}]}\biggl(w_s^{(1,n)}\sum_{d(x,y)=s}f(y) \biggr).$
A similar calculation leads us to get
$ \ds
G_n*\widetilde{f}(0)=\sum_{s=0}^{[\frac{n-1}{2}]}\biggl(w_s^{(2,n)}\sum_{d(x,y)=s}g(y) \biggr).$
Hence the desired result follows.
\end{proof}

We conclude this section with the following observation:
Let $x_0$ be a fixed but arbitrary reference point in $V(\tree)$. Let $\Gamma= Aut(\tree)$ be
the automorphism group of $\tree$ and $H= \{ \varphi \in Aut(\tree): \varphi(x_0)=x_0 \}$
be the stabilizer of $x_0$. Since $\tree$ is vertex-transitive, so the orbit of $x_0$, namely
$\{\varphi(x_0): \varphi \in Aut(\tree)\}$ is equal to $V(\tree)$. By the Orbit-Stabilizer
theorem,  we can identify $V(\tree)$ with $\Gamma/H$. If we choose
$S=\{\varphi \in \Gamma: \varphi(x_0)\sim x_0\}$,  it can be verified that
$\tree$ can also be identified with the coset graph $G=Coset(\Gamma,H, S)$. It is
interesting to observe that in this case, the subgroup  $H$ is an infinite subgroup
of the non-abelian group~$\Gamma$. Therefore, in this section, we have solved the heat
and wave equations on a class of graphs which can be identified with Cayley graphs
as well as coset graphs, whenever the associated group is a non-abelian group.



\end{document}